\documentclass[intlimits,12pt,reqno]{amsart}
\usepackage{latexsym,amsthm,amsmath,amssymb,mathrsfs,mathtools}
\usepackage[T1]{fontenc}
\usepackage[utf8]{inputenc}
\usepackage[mathscr]{eucal}
\usepackage{enumerate}
\usepackage{enumitem}

\let\oldmarginpar\marginpar
\renewcommand{\marginpar}[2][rectangle,draw,text width= 2cm,rounded corners]{
    \oldmarginpar{
    \scriptsize \tikz \node at (0,0) [#1]{#2};}
    }
\usepackage{tikz}
\usetikzlibrary{arrows,calc,patterns,cd}
\usepackage{wrapfig}

\usepackage{hyperref}
\usepackage{xcolor}
\hypersetup{
    colorlinks=true,
    linkcolor=blue,
    citecolor=blue,
    urlcolor=blue
}

plus 6pt minus 12pt
\def\mvint_#1{\mathchoice
          {\mathop{\vrule width 6pt height 3 pt depth -2.5pt
                  \kern -9pt \intop}\limits_{\kern -3pt #1}}%
          {\mathop{\vrule width 5pt height 3 pt depth -2.6pt
                  \kern -6pt \intop}\nolimits_{#1}}%
          {\mathop{\vrule width 5pt height 3 pt depth -2.6pt
                  \kern -6pt \intop}\nolimits_{#1}}%
          {\mathop{\vrule width 5pt height 3 pt depth -2.6pt
                  \kern -6pt \intop}\nolimits_{#1}}}

\newcommand{\INT}{\operatorname{int}}
\newcommand{\epi}{\operatorname{epi}}

\newcommand{\bbbr}{\mathbb R}

\newcommand{\N}{\mathbb N}

\newcommand{\R}{\mathbb R}

\newcommand{\eps}{\varepsilon}

\def\diam{\operatorname{diam}}

\newtheorem{theorem}{Theorem}[section]
\newtheorem*{theorem*}{Theorem}
\newtheorem{lemma}[theorem]{Lemma}

\newtheorem{proposition}[theorem]{Proposition}

\theoremstyle{definition}

\newtheorem{remark}[theorem]{Remark}
\newtheorem*{remark*}{Remark}

\DeclareMathOperator{\Lip}{Lip}
\DeclareMathOperator{\inte}{int}
\DeclareMathOperator{\dist}{dist}

\newcommand*{\loc}{{\mathrm{loc}}}

\usepackage{setspace}
\usepackage{soul} 

\makeatletter
\renewcommand{\tocsection}[3]{%
  \indentlabel{\@ifnotempty{#2}{\bfseries\ignorespaces#1 #2\quad}}\bfseries#3}
\renewcommand{\tocsubsection}[3]{%
  \indentlabel{\@ifnotempty{#2}{\ignorespaces#1 #2\quad}}#3}

\newcommand\@dotsep{4.5}
\def\@tocline#1#2#3#4#5#6#7{\relax
  \ifnum #1>\c@tocdepth 
  \else
    \par \addpenalty\@secpenalty\addvspace{#2}%
    \begingroup \hyphenpenalty\@M
    \@ifempty{#4}{%
      \@tempdima\csname r@tocindent\number#1\endcsname\relax
    }{%
      \@tempdima#4\relax
    }%
    \parindent\z@ \leftskip#3\relax \advance\leftskip\@tempdima\relax
    \rightskip\@pnumwidth plus1em \parfillskip-\@pnumwidth
    #5\leavevmode\hskip-\@tempdima{#6}\nobreak
    \leaders\hbox{$\m@th\mkern \@dotsep mu\hbox{.}\mkern \@dotsep mu$}\hfill
    \nobreak
    \hbox to\@pnumwidth{\@tocpagenum{\ifnum#1=1\bfseries\fi#7}}\par
    \nobreak
    \endgroup
  \fi}
\AtBeginDocument{%
\expandafter\renewcommand\csname r@tocindent0\endcsname{0pt}
}
\def\l@subsection{\@tocline{2}{-5pt}{2.5pc}{5pc}{}}
\makeatother

\title[Approximation of convex bodies]{$\mathbf{C^2}$-Lusin approximation of strongly convex bodies}

\author[Azagra]{Daniel Azagra}
\address{Daniel Azagra, \newline \indent Department of Mathematical Analysis and Applied Mathematics, \newline \indent Universidad Complutense de Madrid,  28040 Madrid, Spain}
\email{azagra@mat.ucm.es}
\thanks{D.A. was supported by grant PID2022-138758NB-I00}

\author[Drake]{Marjorie Drake}
\address{Marjorie Drake, \newline \indent Department of Mathematics, Massachusetts Institute of Technology \newline \indent 
77 Massachusetts Ave., Cambridge, MA 02139}
\email{mkdrake@mit.edu}
\thanks{M.D. was supported by NSF Award No. 2103209}

\author[Haj\l{}asz]{Piotr Haj\l{}asz}
\address{Piotr Haj\l{}asz,\newline \indent Department of Mathematics, University of Pittsburgh, \newline \indent 301 Thackeray Hall, Pittsburgh,
Pennsylvania 15260}
\email{hajlasz@pitt.edu}
\thanks{P.H. was supported by NSF grant  DMS-2055171}

\subjclass[2020]{Primary; Secondary}

\begin{document}

\begin{abstract}
We prove that, if $W \subset \R^n$ is a locally strongly convex body (not necessarily compact), then for any open set $V \supset \partial W$ and $\eps>0$,
there exists a $C^2$ locally strongly convex body $W_{\eps, V}$ such that $\mathcal{H}^{n-1}(\partial W_{\eps, V}\triangle\,\partial W)<\varepsilon$ and $\partial W_{\eps, V}\subset V$. Moreover, if $W$ is strongly convex, then $W_{\eps, V}$ is strongly convex as well.
\end{abstract}

\maketitle

\section{Introduction}
\label{intro}

The aim of this note is to prove the following result.
\begin{theorem}\label{geometric corollary}
Let $W \subset \R^n$ be a locally strongly convex body (not necessarily compact), $\varepsilon>0$,  and the set $V \supset \partial W$ be open. There exists a $C^2$ locally strongly convex body $W_{\eps, V}$ such that $\mathcal{H}^{n-1}(\partial W_{\eps, V}\triangle\,\partial W)<\varepsilon$ and $\partial W_{\eps, V}\subset V$. Moreover, if $W$ is strongly convex, then $W_{\eps, V}$ is strongly convex as well.
\end{theorem}

Here, $\mathcal{H}^{n-1}$ denotes the $(n-1)$-dimensional Hausdorff measure, and $A\triangle B$ is the symmetric difference of the sets $A$ and $B$, that is, 
$A\triangle B :=(A\setminus B)\cup (B\setminus A)$.
Throughout this paper, we say that $W\subset \R^n$ is a \emph{convex body} if it is closed, convex, and has nonempty interior; if its boundary $\partial W$ can be represented locally (up to a suitable rotation) as the graph of a strongly convex function, then we say $W$ is a {\em locally strongly convex body}.
We say that $W$ is a {\em strongly convex body} if it is a compact locally strongly convex body. One can prove that $W$ is a strongly convex body if and only if it is the intersection of a family of closed balls of the same radius; see Proposition~\ref{characterizations of strong convexity of compact bodies} for this and other equivalent characterizations of strongly convex bodies.
Note that the epigraph of a strongly convex function is never a strongly convex body (though it is always a locally strongly convex body). However, if $u:\R^n\to\R$ is locally strongly convex and coercive,
then for every $t>\min_{x\in\R^n}\{u(x)\}$ the level set $u^{-1}((-\infty, t])$ is compact and locally strongly convex, hence also strongly convex; again,
see Proposition~\ref{characterizations of strong convexity of compact bodies}.

A function $u:U\to\R$ defined on an open convex set is {\em strongly convex} if there is $\eta>0$, such that $u(\cdot)-\frac{\eta}{2}|\cdot|^2$ is convex (in which case we say that $u$ is $\eta$-strongly convex). 
Note that, if $u$ is of class $C^2$, then this is equivalent to saying that, for all $x$, the minimum eigenvalue of $D^2 u(x)$ is greater than or equal to $\eta$. We say that $u$ is locally strongly convex if for every $x\in U$ there is $r_x>0$, such that the restriction of $u$ to the open ball $B(x,r_x)$ is strongly convex.

Theorem \ref{geometric corollary} was stated without proof in \cite{ADH} as a corollary to the main result of that paper, which we recall next. 

Let $\mathcal{G}_{u}$ represent the graph of a function $u:U \to \R$, where $U \subset \R^n$. 

\begin{theorem}[See \cite{ADH}]\label{ADHthm}
Let $U\subseteq\R^n$ be open and convex, and $u:U\to\bbbr$ be locally strongly convex. Then for every $\eps_o>0$ and for every continuous function $\varepsilon:U\to (0, 1]$ there is a locally strongly convex function 
$v\in C^2(U)$, such that
\begin{enumerate}
\item[(a)] $|\{x\in U:\, u(x)\neq v(x)\}|<\eps_o$;
\item[(b)] $|u(x)-v(x)|<\varepsilon(x)$ for all $x\in U$;
\item[(c)] $\mathcal{H}^{n}\left( \mathcal{G}_{u} \triangle \mathcal{G}_{v}\right)<\eps_o$.
\end{enumerate}
Also, if $u$ is $\eta$-strongly convex on $U$, then for every $\widetilde{\eta}\in (0, \eta)$ there exists such a function $v$ which is $\widetilde{\eta}$-strongly convex on $U$.
\end{theorem}

Part (a) of this result says that we can approximate a locally strongly convex function by a $C^2$ locally strongly convex function in the Lusin sense. For motivation and background about this kind of approximation we refer the reader to the introductions of the papers \cite{AH, ADH}.

The rest of this note is organized as follows. In Section 2 we review some basic facts of convex analysis, provide multiple characterizations of strongly convex bodies, and detail useful technical estimates for the metric projection onto a compact convex body and onto the boundary of a $C^{1,1}$ {convex} body.
For more details and omitted proofs regarding convex functions and convex bodies we refer to \cite{HUL,Rockafellar,S}. While most of the results of Section 2 are well known, some of the equivalent conditions in Proposition~\ref{characterizations of strong convexity of compact bodies} are new, and Lemma~\ref{T1} is new. In Section 3, we complete the proof of Theorem \ref{geometric corollary}.

\section{Preliminaries for the proof of Theorem~\ref{geometric corollary}.}
\label{S5}

Every closed convex set $W\subset\R^n$ is the intersection of all closed half-spaces that contain $W$. In fact, for every $x\in \partial W$ there is a half-space $H_x$ such that $W\subset H_x$ and $x\in T_x\cap W$, where $T_x=\partial H_x$. The hyperplane $T_x$ is called a {\em hyperplane supporting} $W$ at $x$. For every $x\in \partial W$, there is a hyperplane supporting $W$ at $x$, but such a hyperplane is not necessarily unique. We define the {\em normal cone} to $W$ at $x$ as the set of all vectors perpendicular to some supporting hyperplane of $W$ at $x$ and pointing outside $W$: $$N_W(x):=\{\zeta\in\R^n : \langle \zeta, y-x\rangle\leq 0 \textrm{ for all } y\in W\}.$$
Because there must be a hyperplane supporting $W$ at $x$ for every $x \in \partial W$, we have $N_W(x) \neq\emptyset$ for every $x\in\partial W$.

It follows
that given an open convex set $U\subset\R^n$ and a convex function $f:U\to\R$, we have that for every $x\in U$ there is $v\in\R^n$ such that $f(y)\geq f(x)+\langle v,y-x\rangle$ for all $y\in U$. 
Indeed, on the right hand side we have an equation of a supporting hyperplane of the convex {\em epigraph} $\operatorname{epi}(f)=\{(x,t)\in U\times\R:\, x\in U,\ t\geq f(x)\}$.
The (nonempty) set of all such $v$ is denoted by $\partial f(x)$ and called the {\em subdifferential} of $f$ at $x$:
$$
\partial f (x): = \{v \in \R^{n}: f(y)\geq f(x)+\langle v,y-x\rangle \text{ for all }y\in U\}.
$$
For every $\xi\in \partial f(x)$,  we have that $n_\xi \in \R^{n+1}$ defined by
\begin{equation}\label{normal vector}
n_{\xi}:=\frac{1}{\sqrt{1+|\xi|^2}}\left(\xi, -1\right)
\end{equation}
is a unit normal vector to $\operatorname{epi}(f)$ at $(x, f(x))$ that points outside $\operatorname{epi}(f)$; that is,
$$
n_{\xi}\in N_{\epi(f)}(x,f(x))\cap \mathbb{S}^{n}.
$$

The convex function $f$ is differentiable at a point $x_0$ if and only if $\partial f(x_0)$ is a singleton, in which case we have $\partial f(x_0)=\{\nabla f(x_0)\}$, meaning that the tangent hyperplane to the graph of $f$ at $x_0$ is the unique hyperplane supporting the epigraph of $f$ at $(x_0,f(x_0))$. Convex functions are locally Lipschitz continuous, 
and it easily follows that if $f$ is Lipschitz continuous with Lipschitz constant $L$ in a neighborhood of $x$ and $\xi\in\partial f(x)$, then 
\begin{equation}
\label{eq1}
|\xi|\leq L.
\end{equation}
Also, it follows from Rademacher's theorem that convex functions
are differentiable almost everywhere, so $\partial f(x)=\{\nabla f(x)\}$ for almost every $x\in U$.

The following lemma is well known; for a proof see, for instance, \cite[Lemma 3.4]{ADH}.
\begin{lemma}
\label{T16}
Let $u:U\to\R$ be a convex function defined on an open convex set $U\subseteq\R^n$. Then $u$ is $\eta$-strongly convex if and only if 
\begin{equation}
\label{eq39}
u(y)\geq u(x)+\langle\xi,y-x\rangle +\frac{\eta}{2}|y-x|^2
\end{equation}
for all $x,y\in U$ and $\xi\in \partial u(x)$.
\end{lemma}

\begin{remark}
\label{remark on characterization of strong convexity}
The proof of $(\Leftarrow)$ in \cite[Lemma 3.4]{ADH} also shows that if \eqref{eq39} holds for all $x, y\in U$ and {\em some } $\xi\in \partial u(x)$, then $u$ is $\eta$-strongly convex, and therefore, by the proof of $(\Rightarrow)$, \eqref{eq39} is also true for {\em all } $\xi\in\partial u(x)$.
\end{remark}

For any convex body $W\subset \R^n$ with $0\in\INT(W)$, the {\em Minkowski functional} (also known as {\em gauge}) of $W$ is the map $\mu_W:\R^n \to [0, \infty)$ defined by
$$
\mu_W(x):=\inf\Big\{\lambda>0 : \frac{1}{\lambda}x\in W\Big\}.
$$
The Minkowski functional is a positively homogeneous, subadditive convex function such that $\mu_W^{-1}([0,1])=W$ and $\mu_W^{-1}(1)=\partial W$. Because $0 \in \INT(W)$, there exists $\varepsilon>0$ such that $B(0, \varepsilon) \subset W$; hence, for all $x \in \R^n\setminus\{0\}$, $\frac{\varepsilon x}{2 |x|} \in W$. Thus, $$|\mu_W(x) - \mu_W(y)| \leq \max \{ \mu_W(x-y), \mu_W(y-x) \} \leq \frac{2}{\eps} |x-y|,$$ implying $\mu_W$ is Lipschitz.

As a consequence of the implicit function theorem and the positive homogeneity of $\mu$, we have $\partial W$ is a $1$-codimension submanifold of class $C^k(\R^n)$ if and only if $\mu_W$ is $C^k$ on $\R^n\setminus \mu_W^{-1}(0)$. Note that if $W$ is compact, then $\mu_W^{-1}(0)=\{ 0\}$. 

\begin{lemma}
\label{intersection of selection of supporting halfspaces}
Given a convex body $W\subset\R^n$, for any selection
$$\partial W\ni z\to \zeta(z)\in N_W(z)\cap \mathbb{S}^{n-1},$$  we have
\begin{equation}
\label{intersection of selection of supporting halfspaces equation}
W=\bigcap_{y\in\partial W}\{ x\in \R^n : \langle \zeta(y), x-y\rangle\leq 0\}.
\end{equation}
\end{lemma}
\begin{proof}
Since  $\zeta(y)$ is an outward unit normal vector to $W$ at $y$, the halfspace $H_y^{-}:=\{x: \langle \zeta(y), x-y\rangle\leq 0\}$ contains $W$ for every $y\in\partial W$, so we have that $W\subseteq V:=\bigcap_{y\in\partial W}H_y^{-}$. If $W\neq V$, since both $W$ and $V$ are convex bodies, and we already know that $W\subseteq V$, we must have $y\in \INT(V)$ for some $y\in\partial W$. But then $y\in\INT(H_y^{-})=\{x\in\R^n: \langle \zeta(y), x-y\rangle< 0\}$, which is absurd.
\end{proof}

Most of the equivalences provided by the following result are well known (see \cite{Vial} for instance), except, perhaps, for condition (d)
and the definition of a locally strongly convex body as a closed convex set whose boundary can be locally represented as the graph of a strongly convex function (which is not standard). We provide a complete proof for the reader's convenience.
\begin{proposition}
\label{characterizations of strong convexity of compact bodies}
Let $W\subset \R^n$ be a compact convex body satisfying $0\in\INT(W)$, and let $\mu:\R^n \to [0,\infty)$ denote the Minkowski functional of $W$. The following statements are equivalent:
\begin{enumerate}
\item[(a)] $W$ is strongly convex. 
\item[(b)] There is $R>0$, such that for every $x\in\partial W$, there is a closed ball $\overline{B}(y,R)$, such that $W\subset\overline{B}(y,R)$ and $x\in\partial\overline{B}(y,R)$.
\item[(c)] $W$ is the intersection of a family of closed balls of the same positive radius.
\item[(d)] $\mu^2$ is strongly convex.
\item[(e)] There exists a coercive, locally strongly convex function $g:\R^n\to\R$ such that $W= g^{-1}((-\infty, t])$ for some $t\in\R$ with $t>\min_{x\in\R^n}g(x)$.
\end{enumerate}
\end{proposition}
\begin{proof}
$(a)\Rightarrow (b)$: 

Let $x\in\partial W$, because $W$ is locally strongly convex, $\partial W$ in a neighborhood of $x$ is the graph of a strongly convex function. By translation and rotation if necessary, we can find $r_x>0$ and a function $g_x:B^{n-1}(0,2r_x)\to\R$ that satisfies $g_x(0)=0$, $g_x$ is $\eta_x$-strongly convex, $L_x$-Lipschitz and 
$$
W_{x,2r_x}:=\{(t,g_x(t)):\, t\in B^{n-1}(0,2r_x)\}\subset\partial W,
$$
where the coordinates on the right hand side for $\R^{n-1}\times\R$ depend on $x$. 
In this coordinate system $x$ is represented as $(0,0)$.
More generally, for $s\in (0,2r_x]$ we define
$$
W_{x,s}:=\{(t,g_x(t)):\, t\in B^{n-1}(0,s)\}.
$$
Because $\partial W$ is compact, $\partial W\subset\bigcup_{x\in\partial W}W_{x,r_x}$ has a finite subcover, that is $\partial W\subset\bigcup_{j=1}^m W_{x_j,r_{x_j}}$.

Let $z \in \partial W$. Then there exists $j \in \{1, ..., m\}$ and $x \in B^{n-1}(0, r_j)$ such that $z=(x,g_j(x))$, where we let $g_j$ (resp.\ $r_j$) stand for $g_{x_j}$ (resp.\ $r_{x_j}$). We will show there exists $R>0$ such that for all $j \in \{1, ..., m\}$ and $x \in B^{n-1}(0,r_j)$, and any $\xi\in\partial g_j(x)$
we have 
\begin{equation}
\label{W is contained in suitable balls2}
W \subseteq \overline{B}\left( (x, g_j(x)) -R\, n_{\xi}, R\right),
\end{equation}
where $n_{\xi}$ is defined in \eqref{normal vector}, implying at once $W \subseteq \overline{B}\left( z -R\, n_{\xi}, R\right)$ and $z \in \partial B\left( z -R\, n_{\xi}, R\right)$, and thus (b) holds. Proceeding, our aim is to prove \eqref{W is contained in suitable balls2}.

For $j \in 1, ..., m$, let $\eta_j: = \eta_{x_j}$ and $L_j:=L_{x_j}$.
Let $L, \eta, r, r_0 >0$ be
\begin{align*}
    &L: = \max \{ L_1, ..., L_m\}, \\
    &\eta: = \min \{ \eta_1, ..., \eta_m \}, \\
    &r: = \max \{ r_1, ..., r_m\}, \text{ and}\\
    &r_0 : = \min \{ r_1, ..., r_m \}.
\end{align*}
Let $R>0$ be 
\begin{align}
\label{eq22}
R:= \sqrt{1+L^2} \max\left\{ \frac{1}{\eta}{\left( 1+\frac{\eta^2}{4} r^2 +L^2 +\eta Lr\right)},  \diam(W), \, \frac{\diam(W)^2}{\eta r_0^2} \right\}.
\end{align} 

To verify \eqref{W is contained in suitable balls2} holds with $R$ as in \eqref{eq22}, fix $j\in\{1, ..., m\}$, $x\in B^{n-1}(0, r_j)$, and $\xi\in\partial g_j(x)$. Let $(y,s)\in W$ in coordinates provided by $g_j$. Then we want to show
\begin{align}
    \left| (y, s)-(x, g_j(x)) +R n_\xi \right|^2\leq R^2, \label{aim}
\end{align}
where $n_{\xi}$ is defined in \eqref{normal vector}. We consider two cases.  

{\bf Case 1.} Suppose $|y-x|<r_j$. Since $W\cap\{ (y,s) : |y-x|<r_j\}\subset\{(y,s) : s\geq g_j(y)\}$ and $g_j$ is $\eta$-strongly convex, {Lemma~\ref{T16} yields}
\begin{equation}
\label{eq32}
g_j(x)+\langle \xi, y-x\rangle  +\frac{\eta}{2}|y-x|^2\leq g_j(y)\leq s. 
\end{equation}
Note that $s-g_j(x) \leq \diam(W)$ and in light of \eqref{eq1}, $|\xi|\leq L$, so our choice of $R$ in \eqref{eq22} yields
$$
\frac{R}{\sqrt{1+|\xi|^2}}+g_j(x)-s\geq \frac{R}{\sqrt{1+L^2}}-\diam(W)\geq 0.
$$
In combination with \eqref{eq32},
\begin{equation*}
\label{first part of case 12}
\left( \frac{R}{\sqrt{1+|\xi|^2}}+g_j(x)-s\right)^2\leq
\left( \frac{R}{\sqrt{1+|\xi|^2}}-\langle \xi, y-x\rangle  -\frac{\eta}{2}|y-x|^2\right)^2.
\end{equation*} 
Using \eqref{normal vector} we get
\begin{align}
    \big| (y, s)-(x, g_j(x)) +Rn_\xi \big|^2 &= \bigg| (y, s)-(x, g_j(x)) +\frac{R}{\sqrt{1+|\xi|^2}}\left(\xi, -1\right)\bigg|^2 \nonumber \\
    &= \left|y-x +\frac{R\xi}{\sqrt{1+|\xi|^2}}\right|^2 + \left|s - g_j(x)-\frac{R}{\sqrt{1+|\xi|^2}}\right|^2 \nonumber \\
    &\leq |y-x|^2+\frac{2R}{\sqrt{1+|\xi|^2}}\langle \xi, y-x\rangle +\frac{R^2 |\xi|^2}{1+|\xi|^2} \nonumber \\
    &\quad +\left( \frac{R}{\sqrt{1+|\xi|^2}} -\langle \xi, y-x\rangle-\frac{\eta}{2}|y-x|^2  \right)^2. \label{e1}
\end{align}
Expanding the square in \eqref{e1} and noting by the Cauchy-Schwarz inequality $\langle \xi, y-x \rangle^2 \leq |\xi|^2|y-x|^2$, we have:
\begin{align}
    \big| (y, s)-&(x, g_j(x))+Rn_\xi \big|^2 \nonumber \\
    &\leq |y-x|^2 +R^2\frac{{1+|\xi|^2}}{1+|\xi|^2} + \frac{\eta^2}{4} |y-x|^4 + \langle \xi, y-x\rangle^2 \nonumber \\
    &\quad +\eta\langle \xi, y-x\rangle |y-x|^2 -\frac{R\eta |y-x|^2}{\sqrt{1+|\xi|^2}} \nonumber \\
    &\leq R^2 + |y-x|^2\left(1+\frac{\eta^2}{4}|x-y|^2+ |\xi|^2  +\eta\langle \xi, y-x\rangle -\frac{R\eta}{\sqrt{1+|\xi|^2}}\right).  \label{e2}
\end{align}
To see the quantity in parentheses is bounded by $0$, apply the Cauchy-Schwarz inequality to the inner product, the bounds $|\xi| \leq L$ and $|x-y| \leq r_j \leq r$, and the lower bound on $R$ in \eqref{eq22} to estimate:
\begin{align*}
    1+\frac{\eta^2}{4}|x-y|^2+ |\xi|^2  +\eta\langle \xi, y-x\rangle \leq 1+ \frac{\eta^2}{4}r^2 +L^2 +\eta Lr \overset{\eqref{eq22}}{\leq} \frac{R\eta}{\sqrt{1+|\xi|^2}}.
\end{align*}
Subsituting this into \eqref{e2}, we see $\big| (y, s)-(x, g_j(x))+R  n_{\xi} \big|^2 \leq R^2$, as desired. 

{\bf Case 2.} Now suppose $|y-x|\geq r_j$. Using \eqref{normal vector} and $|n_{\xi}|= 1$, we estimate 
\begin{align}
     \big| (y, s)-(x, g_j(x))+R n_{\xi} \big|^2
     & =R^2 + \left|(y,s) - (x,g_j(x))\right|^2 +2R\frac{\langle y-x, \xi\rangle +g_j(x) -s}{\sqrt{1+|\xi|^2}} \nonumber\\
     & \leq R^2 + \diam(W)^2 +2R\frac{\langle y-x, \xi\rangle +g_j(x) -s}{\sqrt{1+|\xi|^2}}\nonumber\\
     &\leq R^2 +\frac{2R}{\sqrt{1+|\xi|^2}} \bigg( \frac{r_j^2 \eta}{2} + \langle y-x, \xi\rangle +g_j(x) -s \bigg), \label{e3}
\end{align}
where the last inequality follows because the choice of $R$ in \eqref{eq22} and $|\xi| \leq L$ ensure $\diam(W)^2\leq {r_j^2 \eta} \frac{R}{\sqrt{1+|\xi|^2}}$. {Since} $(y,s)\in W$ and $|y-x|\geq r_j$, we have
\begin{equation}
\label{e4}
s\geq g_j(x)+\langle \xi, y-x\rangle +\frac{\eta r_j^2}{2}.
\end{equation}
To see this, 
let $u=r_j(y-x)/|y-x|$, so $x+u\in B^{n-1}(0,2r_j)$ and convexity of $W$ along with Lemma~\ref{T16} imply
$$
\frac{s-g_j(x)}{|y-x|}\geq \frac{g_j(x+u)-g_j(x)}{r_j}\geq\frac{\langle\xi,u\rangle+\frac{\eta}{2}|u|^2}{r_j}=
\Big\langle\xi,\frac{y-x}{|y-x|}\Big\rangle+
\frac{\eta}{2}r_j
$$
and \eqref{e4} follows, because $|y-x|\geq r_j$.

Substituting \eqref{e4} into \eqref{e3}, we conclude
$
\big| (y, s)-(x, g_j(x))+R n_{\xi} \big|^2 \leq R^2.
$
The proof of \eqref{aim} and thus $(a)\Rightarrow (b)$ is complete.

$(b)\Rightarrow (c)$: By assumption, there is $R>0$ such that for every $y\in\partial W$ there exists $x_y\in\R^n$ so that $y\in \partial B(x_y, R)$ and 
\begin{equation}\label{W contained in intersection of balls}
W\subseteq \overline{B}(x_y, R).
\end{equation}
This implies that 
$$
\zeta(y):=\frac{1}{|y-x_y|}(y-x_y) \in N_W(y)\cap\mathbb{S}^{n-1}
$$
for every $y\in\partial W$, and from Lemma \ref{intersection of selection of supporting halfspaces} we deduce that
$$
\bigcap_{y\in\partial W}\overline{B}(x_y, R) =\bigcap_{y\in\partial W}\overline{B}(y-R\zeta(y), R)\subseteq \bigcap_{y\in \partial W}\{x: \langle \zeta(y), x-y\rangle\leq 0\}= W.
$$
In combination with \eqref{W contained in intersection of balls}, we have $\bigcap_{y\in\partial W}\overline{B}(x_y, R) = W$.

$(c)\Rightarrow (d)$: For some nonempty set $\mathcal{A}\subset\R^n$, we may write $W=\bigcap_{a\in \mathcal{A}}B_a$, where $B_a:=\overline{B}(a, R)$, $R>0$.  Since $0\in\INT(W)$ we have $|a|<R$ for all $a\in\mathcal{A}$, and in fact there exists $r>0$ such that
$$
\overline{B}(0, r)\subseteq W\subseteq B_a\subseteq \overline{B}(0, 2R),
$$ 
which implies that
\begin{equation}
|a|\leq R-r \textrm{ for all } a\in\mathcal{A},
\end{equation}
and also that
\begin{equation}\label{inclusion to inequality estimates for gauges}
\frac{1}{2R}|x|\leq\mu_a(x)\leq\frac{1}{r}|x| \textrm{ for all } x\in\R^n,
\end{equation}
where, for any $a\in\mathcal{A}$, 
$$
\mu_{a}(x):=\inf\Big\{\lambda>0 : \frac{x}{\lambda}\in B_a\Big\}
$$
is the Minkowski functional of the ball $B_a$ (with respect to the origin, not necessarily the center $a$ of $B_a$). Since $W=\bigcap_{a\in\mathcal{A}}B_a$,  we have for $x \in \R^n$,
$$\mu(x) = \inf \left\{ \lambda >0: \frac{x}{\lambda} \in\bigcap_{a \in A} B_a \right\} \leq  \sup_{a \in A} \inf \left\{ \lambda >0: \frac{x}{\lambda} \in B_a \right\} = \sup_{a \in A} \mu_a(x)
$$ 
and $W \subset B_a$ implies $\mu_a(x) \leq \mu(x)$ for all $a \in A$. Thus, for $x \in \R^n$, $\mu(x)=\sup_{a\in\mathcal{A}}\mu_a(x)$, and $\mu^2: \R^n \to [0, \infty)$ satisfies
$$
\mu^2(x)=\sup_{a\in\mathcal{A}}{\mu_a}^2(x).
$$
Because the supremum of a family of $\eta$-strongly convex functions is $\eta$-strongly convex, to prove $\mu^2$ is strongly convex, we need only show that there exists $\eta>0$ such that {$\mu_a^2$} is $\eta$-strongly convex for all $a\in\mathcal{A}$.

A straightforward calculation yields
$$
\mu_a(x)=\frac{-\langle x, a \rangle +\sqrt{ \langle x, a \rangle^2 +k_a|x|^2}}{k_a},
$$
where $$k_a:=R^2-|a|^2>0.$$ Differentiating $\mu_a(x)$, for every $x\in\R^n\setminus\{0\}$ we obtain
$$
\nabla\mu_a(x)=\frac{1}{k_a}\left[ -a +\frac{ \langle x, a\rangle a +k_a x}{(\langle x, a\rangle^2+k_a|x|^2)^{1/2}}\right] =\lambda_a(x) (x-\mu_a(x)a),
$$
where $\lambda_a: \R^n\setminus\{0\} \to \R$ is defined by
$$
\lambda_a(x) :=\frac{1}{\sqrt{ \langle x, a\rangle^2 +k_a |x|^2}}.
$$
Then we have
$$
D^2 \mu_a(x)=\nabla\lambda_a(x) \otimes (x-\mu_a(x) a) +\lambda_a(x) \left( I -a\otimes \nabla\mu_a(x)\right),
$$
where $I$ denotes the identity operator. Since $D^2\mu_a(x)$ is symmetric, we also have
$$
D^2\mu_a(x)= (x-\mu_a(x) a)\otimes \nabla\lambda_a(x) +\lambda_a(x) \left( I - \nabla\mu_a(x)\otimes a\right).
$$
Thus,
$$
D^2\mu_a^2(x)= 2\nabla\mu_a(x)\otimes \nabla \mu_a(x) +2\mu_a(x) D^{2}\mu_a(x).
$$

Recalling that $\nabla\mu_a(x)=\lambda_a(x) (x-\mu_a(x) a)$, the above expressions tell us that $$v_0:=\frac{\nabla\mu_a(x)}{|\nabla\mu_a(x)|}=\frac{x-\mu_a(x) a}{|x-\mu_a(x) a|}$$ is an eigenvector of both $D^{2}\mu_a(x)$ and $D^2\mu_a^2(x)$ (here we are using the easy facts that, for any vectors $b, c\in\R^n$, we have $b\otimes c(b)=\langle b,c\rangle b$; hence $b$ is an eigenvector of $b\otimes c$, and that any vector is an eigenvector of the identity). Since $\mu_a$ is convex we have $D^2\mu_a\geq 0$, so we estimate 
\begin{align*}
v_0^T D^2\mu_a^2(x)v_0 &\geq v_0^T \big( 2\nabla\mu_a(x)\otimes \nabla \mu_a(x) \big) v_0 \nonumber \\
&= 2\langle \nabla\mu_a(x), v_0\rangle^2= 2|\nabla\mu_a(x)|^2\geq 2\left(\frac{\mu_a(x)}{|x|}\right)^2\geq \frac{1}{2 R^2},
\end{align*}
where in the two last inequalities we used convexity of $\mu_a$ and \eqref{inclusion to inequality estimates for gauges}.
On the other hand, for every $v\in \mathbb{S}^{n-1}$ with $\langle v_0, v\rangle=0$ we also have $\langle x-\mu_a(x)a, v\rangle =0=\langle \nabla\mu_a(x), v\rangle$, hence
\begin{align*}
v^T D^2\mu_a^2(x)v &= v^T \big( 2 \mu_a(x) \lambda_a(x) I  \big) v \nonumber \\ 
&=\frac{2\mu_a(x)}{\sqrt{\langle x, a\rangle^2 +k_a|x|^2}}
\geq \frac{2\mu_a(x)}{\sqrt{|a|^2 |x|^2 +k_a|x|^2}}=
\frac{2\mu_a(x)}{R|x|}\geq \frac{1}{R^2}. 
\end{align*}
Let $\alpha_0$ be the eigenvalue associated to the eigenvector $v_0$, and let $\alpha_1, ... , \alpha_{n-1}$ be the rest of eigenvalues of $D^2\mu_a^2(x)$ (possibly repeated).  Because $D^2\mu_a^2(x)$ is symmetric and $v_0$ is an eigenvector of norm $1$, we can find eigenvectors $v_1, ..., v_{n-1}$ of $D^2\mu_a^2(x)$ with associated eigenvalues $\alpha_1, ... , \alpha_{n-1}$ so that $\{v_0, v_1, ..., v_{n-1}\}$ is an orthonormal basis of $\R^n$. The last two inequalities imply 
\begin{align*}
\alpha_j= v_j^T D^2\mu_a^2(x)v_j \geq\frac{1}{2R^2}   
\end{align*}
for all $j=0, 1, ..., n$, $x\neq 0$.
We deduce that the minimum eigenvalue of $D^2\mu_a^2(x)$ is greater than or equal to $\frac{1}{2R^2}$, and therefore $\mu_a^2$ is $\frac{1}{2R^2}$-strongly convex on any convex subset of $\R^n \setminus \{0\}$. Finally, if $x=0$, since $\mu_a$ is positive homogeneous,  $\mu_a^2$ is $2$-homogeneous, and we compute
$$v^T D^2\mu_a^2(x)v=\frac{d^2}{dt^2}\mu_a(tv)^2|_{t=0}=\frac{d^2}{dt^2}t^2\mu_a(v)^2|_{t=0}=2\mu_a(v)^2\geq\frac{|v|^2}{2R^2}=\frac{1}{2R^2}$$ for every $v\in\mathbb{S}^{n-1}$. We conclude $\mu_a^2$ is $\frac{1}{2R^2}$-strongly convex on all of $\R^n$, and thus $\mu^2(x) = \sup_{a \in A} \mu_a^2(x)$ is $\frac{1}{2R^2}$-strongly convex on $\R^n$.

$(d)\Rightarrow (e)$ is trivial (let $g=\mu_a^2$, $t=1$).

$(e)\Rightarrow (b)$: We assume the function $g:\R^n\to\R$ is locally strongly convex and coercive (and $W= g^{-1}((-\infty, t])$ for some $t>\min_{x\in\R^n}g(x)$), implying $g$ attains a minimum at a unique $x_0 \in \R^n$. We show for $c>g(x_0)$, there exists $R(c)>0$ such that the level set $K_c:=g^{-1}((-\infty, c]) $ satisfies (b), and, therefore, $W= g^{-1}((-\infty, t])$ satisfies (b) with $R = R(t)$. 

Notice the set $K_c$ is compact because $g$ is coercive, and therefore $g$ is $L$-Lipschitz for some $L>0$ on an open set $U\supset K_c$. Since $g$ is locally strongly convex, up to taking a smaller $U$ we may assume  that  $g$ is $\eta$-strongly convex on $U$.  Let $R:=\frac{L}{\eta}$; fix $y\in\partial K_c$ and $\zeta_y \in \partial g(y)$; then $g(y)=c$.  Because $g$ is coercive, $x_0$ must lie in the interior of $K_c$. Since $y\neq x_0$, we have $\zeta_y\neq 0$, and by the strong convexity of $g$ described in \eqref{eq39}, for $x\in K_c$,
$$
c=g(y)\geq g(x)\geq g(y)+\langle \zeta_y, x-y\rangle +\frac{\eta}{2}|x-y|^2,
$$
which implies
$$
\left|x-y +\frac{1}{\eta}\zeta_y\right|^2\leq \frac{|\zeta_y|^2}{\eta^2},
$$
showing that
$$
K_c \subseteq \overline{B}\left(y-\frac{1}{\eta}\zeta_y, \frac{|\zeta_y|}{\eta}\right)\subseteq \overline{B}\left(y-\frac{R}{|\zeta_y|}\zeta_y, R\right),
$$
completing the proof that $K_c$ satisfies (b), and therefore, $W$ satisfies (b) with $R = R(t) = \frac{L}{\eta}$.

$(b)\Rightarrow (a)$: Let $x\in\partial W$. Since $0\in\INT(W)$ there exists $r>0$ such that $\overline{B}(0,r)\subset\INT(W)$. Let $T_x$ be a hyperplane supporting $W$ at $x$ and $\zeta_x\in N_W(x)$ satisfy $\zeta_x$ is perpendicular to $T_x$. Then the ray $\{x-t\zeta_x: t>0\}$ intersects $\INT(W)$, so there are $t_0>0$ and $r>0$ such that $\overline{B}(x-t_0 \zeta_x, r)\subset\INT(W)$. For every $y\in T_x$ with $|y-x|<r$, the ray $\{y-t\zeta_x : t\geq 0\}$ passes through the ball $B(x-t_0 \zeta_x, r)$ and, therefore, intersects $\partial W$ at exactly two points; the first defines a function whose graph coincides with $\partial W$ in a neighborhood of $x$. Precisely, we define the function $g: T_x\cap B(x,r)\to [0, \infty)$ by $g(y):=\min\{t\geq 0: y+t\zeta_x\in\partial W\}$. This function is convex because its graph coincides with $\partial W$ on $U_x$ a neighborhood of $x$,  and $W$ is convex. 
We will show that $g$ is strongly convex, proving (a).

By assumption, there exists $R>0$ such that for every $y\in U_x\cap\partial W$ there is $v_y\in\mathbb{S}^{n-1}$ so that $W\subseteq \overline{B}(y-R v_y, R)$. In particular $v_y\in N_W(y)\cap N_{\overline{B}(y-R v_y, R)}(y)$, and $\partial B(y-Rv_y, R)\cap U_x$ is the graph of a $C^{\infty}$ convex function $f_y: T_x\cap B(x,r)\to \R$ such that 
\begin{equation}
\label{the gradient of fy at y is a subgradient of g at y}
f_y\leq g, \,\, f_y(y)=g(y), \,\, \nabla f_y(y)\in \partial g(y), \textrm{ and } 
v_y=(\xi_y, s_y),
\end{equation}
where
$$
\xi_y:=\frac{1}{\sqrt{1+|\nabla f_y(y)|^2}}\nabla f_y(y), \textrm{ and } s_y=\frac{-1}{\sqrt{1+|\nabla f_y(y)|^2}}.
$$
In the coordinates given by the hyperplane $T_x$ and its normal vector $-\zeta_x$, we have
$$
f_y(z)=g(y)-Rs_y-\sqrt{R^2-|z-y+R\xi_y|^2},
$$
and a straightforward calculation shows that
$$
D^2 f_y(z)= \frac{ \left(z-y+R\xi_y\right) \otimes \left(z-y+R\xi_y\right) +\left(R^2-|z-y+R\xi_y|^2\right) I}{\left(R^2 - |z-y+R\xi_y|^2\right)^{3/2}},
$$
where $I$ is the identity operator.
Clearly,
$$
w_0:=\frac{1}{|z-y+R\xi_y|}\left(z-y+R\xi_y\right)
$$
is an eigenvector of $D^{2}f_y(z)$, and we have
$$
w_0^T D^2 f_y(z)w_0=\frac{R^2}{\left(R^2 - |z-y+R\xi_y|^2\right)^{3/2}}\geq \frac{1}{R}.
$$
On the other hand, for all $w\in \mathbb{S}^{n-1}$ with $\langle w, w_0\rangle=0$ we have $\langle z-y+R\xi_y, w\rangle=0$, so
$$
w^T D^2f_y(z)w=\frac{R^2 - |z-y+R\xi_y|^2}{\left( R^2 - |z-y+R\xi_y|^2\right)^{3/2}}=\frac{1}{\left( R^2 - |z-y+R\xi_y|^2\right)^{1/2}}\geq \frac{1}{R}.
$$
Hence
$$
\min_{|w|=1}w^T D^2f_y(z)w \geq \frac{1}{R},
$$
and
$f_y$ is $\frac{1}{R}$-strongly convex. Now, by \eqref{the gradient of fy at y is a subgradient of g at y} and Lemma \ref{T16} it follows that
\begin{eqnarray}
& &  g(z)\geq f_y(z)\geq f_y(y)+\langle\nabla f_y(y), z-y\rangle+ \frac{1}{2R}|z-y|^2 \\
& &= g(y)+\langle\nabla f_y(y), z-y\rangle+ \frac{1}{2R}|z-y|^2,
\end{eqnarray}
so by Remark \ref{remark on characterization of strong convexity} we conclude that $g$ is $\frac{1}{R}$-strongly convex too. The proof is complete.
\end{proof}

A self-evident local variant of the proofs of $(a)\Rightarrow (b)$ and $(b)\Rightarrow (a)$ in Proposition~\ref{characterizations of strong convexity of compact bodies} shows the following:

\begin{lemma}\label{equivalent defn of locally strongly convex body}
For any (possibly unbounded) convex body $W\subset\R^n$ the following statements are equivalent:
\begin{enumerate}
\item $W$ is locally strongly convex (in the sense that $\partial W$ is locally, up to a rigid change of coordinates, the graph of a strongly convex function).
\item For every $x\in\partial W$ there exist an open neighborhood $U_x\ni x$ and a number $R_x>0$ such that, for all $y\in U_x\cap \partial W$ there is $v_y\in\mathbb{S}^{n-1}$ such that $W\cap U_x\subset \overline{B}(y-R_x v_y, R)$.
\end{enumerate}
Moreover, if (in appropriate coordinates) $W$ is the epigraph of a convex function $f$ and one of the conditions is satisfied, then $f$ is locally strongly  convex.
\end{lemma}

For any closed convex set $C\subset\R^n$ the metric projection $\pi_C: \R^n\to C$ (defined, for every $x\in\R^n$, as the unique point $\pi(x)\in C$ such that $\dist(x, \pi(x))=\dist(x, C)$\,) is $1$-Lipschitz; see \cite[(3.1.6)]{HUL} for a proof. 
Clearly, $\pi(x)\in\partial C$ if $x\notin\INT(C)$. When the boundary $\partial C$ is of class $C^{1,1}$, a bit more is true: the metric projection onto the (not necessarily convex) boundary $\partial C$ is also well defined and Lipschitz on an open neighborhood of $\partial C$.

For $W\subset \R^n$  satisfying $\partial W$ is of class $C^{1,1}$, let  $n_{\partial W}:\partial W\to \mathbb{S}^{n-1}$ be the outward unit normal vector to $\partial W$. Recall, 
$$
\Lip(n_{\partial W}):= \sup \left\{ \frac{|(n_{\partial W}(x) - n_{\partial W}(y)|}{|x-y|}: x, y \in \partial W, x \neq y\right\}.
$$
\begin{lemma}
\label{metric projection is 2 Lip on a neighborhood of a C11 boundary}
Let $W \subset \R^n$ be a closed convex set with nonempty interior such that $\partial W$ is of class $C^{1,1}$.  Then the metric projection $\pi: \Omega\to\partial W$ is well defined and $2$-Lipschitz, where
$$
\Omega: =\left\{x\in\R^n : d(x, \partial W)<  \frac{1}{2\Lip(n_{\partial W})}  \right\}\cup W^c.
$$ 
\end{lemma}
\begin{proof}
See, for instance, \cite[Theorem 2.4]{AM}, or the references therein.
\end{proof}

We intend to apply the following lemma when $W,V \subset \R^n$ are compact but include the more general result:

\begin{lemma} \label{T1}
Let $W,V$ be (possibly not bounded) convex bodies such that $W\subset V \subsetneq\R^n$, and 
$\mathcal{H}^{n-1}\left(\partial V\setminus\partial W\right)<\infty$. 
Then the projection $\pi_W:\R^n\to W$ maps $\partial V$ \emph{onto} $\partial W$. 
\end{lemma}

\begin{proof} 
We consider two cases:

{\bf Case 1.} Suppose that $\partial V$ does not contain any lines.\footnote{This case was shown in an argument inside the proof of \cite[Theorem 1.6]{ACH}; we reproduce it here for completeness.} We will show for all $x\in \partial W$, there is $z\in \partial V$ such that $\pi_{W}(z)=x$. Let $\nu \in N_W(x)$. It suffices to show that the ray
$
R_x:=\{x+t\nu:\, t\geq 0\}
$
intersects $\partial V$ at some point $z$, implying $\pi_{W}(z)=x$. 

Suppose not; then $R_x \cap \partial V = \emptyset$, implying $R_x\subset\operatorname{int} V$.
Let $T_x \subset \R^n$ be a supporting hyperplane for $W$ at $x \in \partial W$:
$$
T_x:=\{x+v:\langle v,\nu\rangle=0\} 
$$
Then the open half space $H_x:=\{x+v:\langle v,\nu\rangle>0\}$ satisfies $H_x\cap \partial W=\emptyset$.
Because $x\in\operatorname{int}(V)$, there exists $\delta>0$ such that $B(x,2\delta) \cap T_x \subset \inte(V)$. Further, $(B(x,2\delta) \cap T_x)\cup R_x\subset \inte (V)$ and $V$ is convex, so $C_x\subset \inte (V)$, where
$
C_x:=\{p+t\nu:\, p\in \partial B(x,\delta) \cap T_x,\ t>0\}
$
is the side surface of a half-cylinder. 
Since $\partial V$ does not contain a line, the set $V$ does not contain a line.
Therefore, for $p\in R_x$, $v\in S^{n-1}$  satisfying $v$ is parallel to $T_x$, the line
$
L_{p,v}:=\{p+tv:\, t\in\R\}
$
must intersect $\partial V$. Let $A \subset H_x$ be
$$
A:= \bigcup_{p \in R_x, \langle v, \nu \rangle = 0}L_{p, v} \cap \partial V.
$$
Let $\pi$ be the radial projection of $A$ onto $C_x$ along lines $L_{p,v}$; $\pi$ is $1$-Lipschitz and hence $\mathcal{H}^{n-1}(A)\geq \mathcal{H}^{n-1}(\pi(A)) \geq \mathcal{H}^{n-1}(C_x)/2 = \infty$.
Since $H_x\cap \partial W=\emptyset$, we have $A\subset \partial V\setminus \partial W$, implying $\mathcal{H}^{n-1}(\partial V\setminus \partial W)=\infty$, a contradiction.

{\bf Case 2.} Suppose that $\partial V$ contain at least one line. Because $V \neq \R^n$, $V$ must have a cylindrical structure: up to isometry, $V=V_1\times E_0$, where $V_1$ is line-free, convex, and at least $1$-dimensional, and $E_0$ is a linear subspace. By an argument similar to the proof of \cite[Proposition 1.10]{AH}, we deduce that $\partial V=\partial W$ because $\mathcal{H}^{n-1}(\partial V\setminus\partial W)<\infty$. 
\end{proof}

Let us conclude our preliminaries with a restatement of \cite[Cor. 3.10]{ADH}.
\begin{lemma}
\label{T11}
If $u:\bbbr^n\to\bbbr$ is $\eta$-strongly convex, then for every $0<\widetilde{\eta}<\eta$ and every $\eps>0$, there is a 
$\widetilde{\eta}$-strongly convex function $v\in C^{1,1}_{\rm loc}(\bbbr^n)$, such that $v\geq u$ and $|\{x\in\bbbr^n:\, u(x)\neq v(x)\}|<\eps$.
\end{lemma}

\section{Proof of Theorem~\ref{geometric corollary}.}
\label{S6}

We are now fully equipped to proceed with the proof of Theorem~\ref{geometric corollary}. We begin with an auxiliary $C^{1,1}$ version of it.

\begin{lemma}
\label{corollary for convex bodies}
Let $W$ be a compact convex body in $\R^n$,  and $V$ be an open set containing $\partial W$. Then for every $\varepsilon>0$ there exists a compact convex body $W_{\varepsilon}\subseteq W$ of class $C^{1,1}$ such that
$\mathcal{H}^{n-1}\left(\partial W \triangle \partial W_{\varepsilon}\right)< \varepsilon$ and $\partial W_{\varepsilon}\subset V$. Moreover, if $W$ is a strongly convex body, then $W_{\eps}$ is a strongly convex body as well.
\end{lemma}
\begin{proof}
Next, we recall and adapt the proof of \cite[Corollary 1.7]{AH}, or \cite[Theorem 1.4]{ACH}, to our context, showing the bound on the $(n-1)$-dimensional Hausdorff measure of the symmetric difference $\partial W \triangle \partial W_{\varepsilon}$, that $W_\varepsilon$ is a strongly convex body if $W$ is a strongly convex body, and $\partial W_{\varepsilon}\subset V$.

We assume that $0\in \INT(W)$; recall the Minkowski functional of $W$, $\mu: \R^n \to [0, \infty)$ defined by
$$
\mu(x):=\inf\{\lambda> 0  : \frac{x}{\lambda} \in W\},
$$
satisfies $\mu$ is convex and Lipschitz. Let $L$ be the Lipschitz constant of $\mu$. By Lemma \ref{T11},  there exists a convex function $g=g_\varepsilon \in C^{1,1}_{{\textrm{loc}}}(\R^n)$ such that
$$
\left| \{x\in 2 W \, : \, \mu(x)\neq g(x)\}\right|<\frac{\varepsilon}{L}.
$$
Let $C_{1,2},A \subset \R^n$ be
\begin{align*}
&C_{1,2}:=2W\setminus W=\{x\in\R^n : 1< \mu(x)\leq 2\}, \text{ and} \\
&A: =\{x\in C_{1,2} : \mu(x)\neq g(x)\}.
\end{align*}
By the coarea formula for Lipschitz functions (see \cite[Section 3.4.2]{EG}, for instance) we have
$$
\varepsilon> L\, |A|\geq\int_{A}|\nabla \mu(x)|\, dx=\int_{1}^{2}\mathcal{H}^{n-1}\left(A\cap \mu^{-1}(t)\right)\, dt,
$$
implying $|\{s \in (1,2]: \mathcal{H}^{n-1}\left(A\cap \mu^{-1}(s)\right)> \varepsilon \}| < 1$. Because $g \in C^{1,1}_{\loc}(\R^n)$ is convex and does not attain a minimum in $g^{-1}((1,2])$, we have $|\nabla g(x)|>0$ for all $x \in C_{1,2}$. Together, these results imply that there exists a regular value of $g$, $t_{0}\in (1, 2)$, where 
\begin{align}
\mathcal{H}^{n-1}\left(A\cap \mu^{-1}(t_0)\right)< \varepsilon. \label{regval}
\end{align}
Then, we define 
$$
W_{\varepsilon}=\frac{1}{t_0}g^{-1}((-\infty, t_0]).
$$
Because $g$ is convex and $C^{1,1}_{\textrm{loc}}$, and $t_0$ is a regular value of this function, $W_{\varepsilon}$ is a convex body of class $C^{1,1}_{\textrm{loc}}$ with boundary
$$
\partial W_{\varepsilon}=\frac{1}{t_0}g^{-1}(t_0),
$$
implying
$$
t_0(\partial W\setminus \partial W_\varepsilon)=A\cap \mu^{-1}(t_0).
$$
With inequality \eqref{regval}, this yields
$$
\mathcal{H}^{n-1}(\partial W\setminus\partial W_\varepsilon)\leq 
t_{0}^{n-1} \mathcal{H}^{n-1}\left(\partial W\setminus \partial W_{\varepsilon}\right)=
\mathcal{H}^{n-1}\left(A\cap \mu^{-1}(t_0)\right)<\varepsilon.
$$
Since $g\geq \mu$, we have $W_{\eps}\subset W$. In particular, $W_{\eps}$ is compact and, therefore, of class $C^{1,1}$. Because the metric projection $\pi:\R^n\to W_{\eps}$, is $1$-Lipschitz and maps $\partial W$ onto $\partial W_{\eps}$, we also have
$$
\mathcal{H}^{n-1}\left(\partial W_{\eps}\setminus\partial W\right)=
\mathcal{H}^{n-1}\left(\pi\left(\partial W\setminus\partial W_{\eps}\right)\right)\leq
\mathcal{H}^{n-1}\left(\partial W\setminus\partial W_\varepsilon\right)<\varepsilon.
$$
Therefore $\mathcal{H}^{n-1}\left(\partial W \triangle \partial W_\varepsilon\right)<2\varepsilon$.

If we further assume that $W$ is a strongly convex body, then by Proposition~\ref{characterizations of strong convexity of compact bodies}, $\mu^2$ is a strongly convex function, and applying Lemma \ref{T11}, we obtain a strongly convex function $g\in C^{1,1}_{\loc}(\R^n)$ such that $\mu^2\leq g$, and
$$
\left| \{x\in 2 W \, : \, \mu^2(x)\neq g(x)\}\right|<\frac{\varepsilon}{L},
$$
where $L=\Lip(\mu)$. Thus,
$$
\left| \{x\in 2 W \, : \, \mu(x)\neq h(x)\}\right|<\frac{\varepsilon}{L},
$$
where $h: \R^n \to \R$ is defined by $h(x):=|g(x)|^{1/2}$. Because $|\nabla g(x)|>0$ for $x \in C_{1,2}$, we have $h \in C^{1,1}(C_{1,2})$. There exists a regular value of $h$, $t_0 \in (1,2)$ satisfying 
an analog of \eqref{regval}. Let $W_\varepsilon \subset \R^n$ be
$$
W_{\varepsilon}:=\frac{1}{t_0}h^{-1}((-\infty, t_0]);
$$
then, $\frac{1}{t_0}g^{-1}(t_{0}^2)=\frac{1}{t_0}h^{-1}(t_0)=\partial W_{\varepsilon}$. Because $h$ is coercive, by Proposition~\ref{characterizations of strong convexity of compact bodies} (e) $\Rightarrow$ (a), we deduce $W_\varepsilon$ is a strongly convex body. Further, the inequality $\mu^2\leq g$ implies that $W_{\varepsilon}\subset W$. The proof that $\mathcal{H}^{n-1}\left(\partial W \triangle \partial W_\varepsilon\right)<2\varepsilon$ is completed exactly as above.

Finally, given an open set $V\supset \partial W$, we want to show $\partial W_{\varepsilon}\subset V$ if $\varepsilon$ is small enough. Suppose not; then there exists a sequence of $C^{1,1}$ (strongly) convex bodies $(U_k)_{k\in\N}$ such that 
$$
\mathcal{H}^{n-1}\left(\partial W \triangle \partial U_{k}\right)< 1/k \text{ and } U_k\subseteq W \text{ for all } k \in \N.
$$ 
Because $W$ is compact, $V\supset \{x \in \R^n: \dist(x, \partial W) \leq 2r\}$ for some $r>0$. Thus, there is sequence $(z_k)_{k \in \N}$ with $z_k\in\partial U_k$ for each $k \in \N$ such that
$$
\dist(z_k, \partial W)\geq 2r>0
$$ 
Since $(z_k)_{k \in \N} \subset W$, up to taking a subsequence, we may assume that $(z_k)_{k \in \N}$ converges to some $z_0\in W$, and, necessarily, $\dist(z_0, \partial W)\geq 2r>0$. 
Hence, there exists $k_0\in \N$ such that for $k\geq k_0$, we have $B(z_k, r)\subset B(z_0, 2r)\subset W$. Let $H_k$ denote the tangent hyperplane to $\partial U_k$ at $z_k$, and $H_k^{-}$ and $H_k^{+}$ denote the open halfspaces with common boundary $H_k$. Suppose $U_k\subset \overline{H_{k}^{-}}$; observe that the metric projection $\pi:\partial W\to \partial B(z_k, r)$ is $1$-Lipschitz and maps $\partial W\cap H_k^{+}$ onto $\partial B(z_k, r)\cap H_k^{+}$. We deduce
\begin{align*}
\frac{1}{2}\mathcal{H}^{n-1}(\partial B(0,r))&= \mathcal{H}^{n-1}\left(	\partial B(z_k, r) \cap H_k^{+}\right)\\
&\leq 
\mathcal{H}^{n-1}\left(	\partial W\cap H_k^{+}\right)\leq 
\mathcal{H}^{n-1}\left(	\partial W \triangle \partial U_k\right)\leq 1/k 
\end{align*}
for all $k\geq k_0$, which is absurd. Thus for $\varepsilon$ small enough, we must have $\partial W_{\varepsilon}\subset V$.
\end{proof}

\begin{proof}[Proof of Theorem \ref{geometric corollary}]
Let $W \subset \R^n$ be a locally strongly convex body, $\varepsilon>0$,  and the set $V \supset \partial W$ be open. We want to show there exists a $C^2$ locally strongly convex body $W_{\eps, V}$ such that $\mathcal{H}^{n-1}(\partial W_{\eps, V}\triangle\,\partial W)<\varepsilon$ and $\partial W_{\eps, V}\subset V$. Moreover, if $W$ is a strongly convex body, then $W_{\eps, V}$ can be chosen to be a strongly convex body as well. We consider two cases:

\noindent {\bf Case 1. Suppose that $W$ is \underline{not} bounded.} Because $W$ is locally strongly convex, $\partial W$ can be regarded, up to a suitable rotation, as the graph of a convex function $f:U\subseteq\R^{n-1}\to\R$ such that $\lim_{y\in U, |y|\to\infty}f(y)=\infty$ (if $U$ is not bounded) and $\lim_{y\to x}f(y)=\infty$ for every $x\in\partial U$ (if $U\neq\R^{n-1}$); see \cite{AS} for instance.\footnote{We warn the reader that what in this paper we call a locally strongly convex function is called a strongly convex function in \cite{AS}.} According to Lemma \ref{equivalent defn of locally strongly convex body} the function $f$ is locally strongly convex. Hence the result is a straightforward consequence of Theorem~\ref{ADHthm}. (Notice (b) of Theorem~\ref{ADHthm} can be used to ensure $\partial W_{\eps,V} \subset V$.)

\noindent {\bf Case 2. Suppose that $W$ is bounded.} Then $W$ is compact and thus a strongly convex body. By Lemma \ref{corollary for convex bodies}, there exists a strongly convex body $W_{\eps/2} \subseteq W$ of class $C^{1,1}$ such that $\mathcal{H}^{n-1}(\partial W \triangle \partial W_{\eps/2})<\eps/2$. We will prove there exists a $C^2$ strongly convex body $W_{\eps/2,V}$, satisfying $\mathcal{H}^{n-1}(\partial W_{\eps/2} \triangle \partial W_{\eps/2, V})<\eps/2$. Then because 
\begin{align*}
    \partial W  \triangle \partial W_{\eps/2, V} \subset (\partial W  \triangle \partial W_{\eps/2} ) \cup (\partial W_{\eps/2}  \triangle \partial W_{\eps/2, V} ),
\end{align*}
we deduce
\begin{align*}
    \mathcal{H}^{n-1}\left(  \partial W  \triangle \partial W_{\eps/2, V}\right)\leq \mathcal{H}^{n-1}\left( \partial W  \triangle \partial W_{\eps/2}\right)+ \mathcal{H}^{n-1}\left(\partial W_{\eps/2}  \triangle \partial W_{\eps/2, V} \right) < \eps
\end{align*}
Hence, from now on, we assume $W$ is a $C^{1,1}$ strongly convex body.

By Lemma \ref{metric projection is 2 Lip on a neighborhood of a C11 boundary} we know that there exists an open neighborhood $\Omega$ of $\partial W$ such that the metric projection $\pi: \Omega\to\partial W$ is well defined and $2$-Lipschitz.  Without loss of generality we may assume that $V\subset \Omega$ and $0\in \textrm{int}(W)$. Let $\mu: \R^n \to [0,\infty)$ be the Minkowski functional of $W$; recall 
$$
\mu(x)=\inf\{\lambda\geq 0 \, : \, \frac{1}{\lambda} x\in W\}.
$$
The function $\mu$ is convex and Lipschitz on $\R^n$, and of class $C^{1,1}$ on $\R^n\setminus B(0,r)$ for every $r>0$.  
Let $L$ be the Lipschitz constant of $\mu$, and let $R>0$ be large enough so that 
$$ 
2 W\subseteq B(0, R).
$$ 
We may assume our given $\varepsilon$ is in $(0, 1/4)$ and small enough so that
$$
\mu^{-1}\left([1-5\varepsilon, 1+5\varepsilon]\right)\subset V\subset\Omega.
$$
Applying Lemma \ref{characterizations of strong convexity of compact bodies} (a) $\Rightarrow$ (d) to $W$, we deduce $\mu^2$ is strongly convex on $\R^n$. By Theorem~\ref{ADHthm} there exists a strongly convex function $g\in C^{2}(\R^n)$ such that
\begin{align}
\left|\{x\in B(0, R) \, : \, \mu(x)^2\neq g(x)\}\right|<\frac{\varepsilon^2}{(8 L^2R+4\varepsilon/R)2^n} \label{lusbd}
\end{align}
and for all $x \in \R^n$,
\begin{align}
|\mu^2(x) -g(x)|<\eps. \label{ptbd}
\end{align}
Because $\mu$ is $L$-Lipschitz, we have
$$
-\varepsilon\leq g(x)\leq\mu(x)^2+\varepsilon \leq 4(LR)^2+\varepsilon \quad (x\in B(0, 2R)).
$$
Applying \cite[Lemma 3.3]{ADH}, we deduce 
$$
\textrm{Lip}\left(g_{|_{B(0,R)}}\right)\leq \frac{4(LR)^2+2\varepsilon}{R}.
$$
Let $h:\R^n \to \R$ be defined as $h(x):=|g(x)|^{1/2}$; then for $ x\in h^{-1}\left([1, 1+\varepsilon]\right) \subset g^{-1}\left([1, 1+\varepsilon]\right)$,
\begin{align*}
    |\nabla h(x)| &= \frac{|\nabla g(x)|}{2|g(x)|^{1/2}} \\
    &\leq 2L^2R+\varepsilon/R.
\end{align*}
Further from (\ref{ptbd}), for $x \in h^{-1}([1,1+ \eps])$, we have 
\begin{align*}
    &h^2(x) - \eps \leq \mu^2(x) \leq h^2(x) + \eps, \text{ implying} \\
    &1 - \eps \leq \mu^2(x) \leq 1 + 4 \eps, \text{ and thus,}\\
    &1 - \eps \leq \mu(x) \leq 1 + 4 \eps \quad \big(x \in h^{-1}([1, 1+\eps]) \big).
\end{align*}
This shows that 
$$
h^{-1}\left([1, 1+\varepsilon]\right)\subset \mu^{-1}\left([1-5\varepsilon, 1+5\varepsilon]\right)\subset V\subset\Omega.
$$

Now consider the set 
$$
A:=\{x\in h^{-1}([1, 1+\varepsilon
]): \mu(x)^2\neq g(x)\}=\{x\in h^{-1}([1, 1+\varepsilon
]) : \mu(x)\neq h(x)\}.
$$ 
By the coarea formula for Lipschitz functions (see \cite[Theorem 3.10, Section 3.4.2]{EG} for instance) we have
$$
\frac{\varepsilon^2}{2^{n+2}}> \left( 2 L^2R +\varepsilon/R\right) |A| \geq\int_{A}|\nabla h(x)|\, dx=\int_{1}^{1+\varepsilon}\mathcal{H}^{n-1}\left(A\cap h^{-1}(t)\right)\, dt.
$$
This inequality implies that there exists $t_{0}\in (1, 1+\varepsilon)$ such that
$$
\mathcal{H}^{n-1}\left(A\cap h^{-1}(t_0)\right)< \varepsilon/2^{n+2},
$$
and because $g$ is convex and cannot have a minimum in {$g^{-1}((1, 2])$}, the number ${t_0}^2$ is a regular value of $g$.
Then, we define 
$$
W_{\varepsilon}:=\frac{1}{t_0}h^{-1}((-\infty, t_0]).
$$
Since $\partial W_{\eps}=\frac{1}{t_0} h^{-1}(t_0)=\frac{1}{t_0} g^{-1}({t_0}^2)$ is a hypersurface of class $C^2$,  and $h$ is coercive, we apply Proposition \ref{characterizations of strong convexity of compact bodies} to deduce that $W_{\varepsilon}$ is a strongly convex body of class $C^{2}$, and 
$$
t_0(\partial W_{\eps}\setminus \partial W)=A\cap h^{-1}(t_0).
$$
This yields
$$
\mathcal{H}^{n-1}(\partial W_{\eps}\setminus\partial W)\leq 
t_{0}^{n-1} \mathcal{H}^{n-1}\left(\partial W_{\eps}\setminus \partial W\right)=
\mathcal{H}^{n-1}\left(A\cap h^{-1}(t_0)\right)<\varepsilon/2^{n+2}.
$$
Further, 
$$
\partial W_{\varepsilon}\subset \mu^{-1}\left([1-\varepsilon, 1+\varepsilon]\right))\subset V\subset\Omega,
$$
and, consequently, the metric projection $\pi: \partial W_{\eps}\to \partial W$ is well-defined and $2$-Lipschitz. Hence,
$$
\mathcal{H}^{n-1}(\partial W\setminus\partial W_{\eps})\leq 
2^{n-1}\mathcal{H}^{n-1}\left(\partial W_{\eps}\setminus \partial W\right)<\varepsilon/4.
$$
Therefore, we conclude $\mathcal{H}^{n-1}(\partial W \triangle\partial W_{\eps})<\varepsilon$.
\end{proof}



\begin{thebibliography}{00}
\frenchspacing
\setlength{\parskip}{0pt}
\setlength{\itemsep}{1 pt plus 0.5 pt minus 0.5 pt}


\bibitem{ADH}
{\sc Azagra, D., Drake, M.,  Haj\l{}asz, P.:}
$C^2$-Lusin approximation of strongly convex functions.
{\em Invent. Math.} 236 (2024), no. 3, 1055--1082. 

\bibitem{ACH}
{\sc Azagra, D., Cappello, A., Haj\l{}asz, P.:}
A geometric approach to second-order differentiability of convex functions.
{\em  Proc. Amer. Math. Soc. Ser. B} 10 (2023), 382--397. 



\bibitem{AH}
{\sc Azagra, D., Haj\l{}asz, P.:}
Lusin-type properties of convex functions and convex bodies.
{\em J. Geom.\ Anal.} 31 (2021), 11685--11701.



\bibitem{AM}
{\sc Azagra, D., Mudarra, C.:} Prescribing tangent hyperplanes to $C^{1,1}$ and $C^{1, \omega}$ convex hypersurfaces in Hilbert and superreflexive Banach spaces.
{\em J. Convex Anal.} 27 (2020) no. 1, 81104.


\bibitem{AS}
{\sc Azagra, D., Stolyarov, D.:} Inner and outer smooth approximation of convex hypersurfaces. When is it possible? {\em  Nonlinear Anal.} 230 (2023), Paper No. 113225. 


\bibitem{EG}
{\sc Evans, L. C., Gariepy, R. F.:} {\em Measure theory and fine properties of functions.} 
Revised edition. Textbooks in Mathematics. CRC Press, Boca Raton, FL, 2015.

\bibitem{HUL}
{\sc Hiriart-Urruty, J.-B., Lemar\'echal, C.:}
{\em Fundamentals of convex analysis.} Grundlehren Text Editions. Springer-Verlag, Berlin, 2001.


\bibitem{Rockafellar}
{\sc Rockafellar, R.T.:} {\em Convex analysis.} Princeton Mathematical Series, No. 28. Princeton University Press, Princeton, N.J., 1970.

\bibitem{S}
{\sc Schneider, R.:}
{\em Convex bodies: the Brunn-Minkowski theory.}
Second expanded edition. Encyclopedia of Mathematics and its Applications, 151. Cambridge University Press, Cambridge, 2014.

\bibitem{Vial}
{\sc Vial, J.P.:}  Strong convexity of sets and functions. {\em J. Math.\ Econom.} 9 (1982), 187--205.

\end{thebibliography}
\end{document}